\documentclass[12pt,twoside]{amsart}

\usepackage{amsmath,amsthm,amscd,amssymb,mathrsfs,graphicx,amsfonts,mathrsfs}
\usepackage{amsfonts}
\usepackage{amssymb,enumerate}
\usepackage{amsthm}
\usepackage[all]{xy}
\usepackage{hyperref}
\allowdisplaybreaks[4] \footskip=12pt
\renewcommand{\uppercasenonmath}[1]{}

\numberwithin{equation}{section} \theoremstyle{plain}
\newtheorem*{thm*}{Main Theorem}
\newtheorem{thm}{Theorem}[section]
\newtheorem{cor}[thm]{Corollary}
\newtheorem*{cor*}{Corollary}
\newtheorem{lem}[thm]{Lemma}
\newtheorem*{lem*}{Lemma}

\newtheorem*{fact*}{Fact}

\newtheorem*{nota*}{Notation}
\newtheorem{prop}[thm]{Proposition}
\newtheorem*{prop*}{Proposition}

\newtheorem*{rem*}{Remark}

\newtheorem*{observation*}{Observation}

\newtheorem*{exa*}{Example}

\newtheorem*{df*}{Definition}

\newtheorem*{con*}{Construction}

\renewcommand{\geq}{\geqslant}
\renewcommand{\leq}{\leqslant}

\begin{document}
\begin{center}
{\large  \bf  Cofiniteness with respect to extension of Serre subcategories}

\vspace{0.5cm} Xiaoyan Yang\\
Department of Mathematics, Northwest Normal University, Lanzhou 730070,
China
E-mail: yangxy@nwnu.edu.cn
\end{center}

\bigskip
\centerline { \bf  Abstract}
\leftskip10truemm \rightskip10truemm \noindent Let $\mathfrak{a}$ be an ideal of a commutative noetherian ring $R$, $\mathcal{S}$ a Serre
subcategory of $R$-modules satisfying the condition $C_\mathfrak{a}$ and $\mathcal{N}$ the subcategory of finitely
generated $R$-modules. In this paper, we continue the study of $\mathcal{NS}$-$\mathfrak{a}$-cofinite modules with respect
to the extension subcategory $\mathcal{NS}$, show that some classical results of $\mathfrak{a}$-cofiniteness hold for
$\mathcal{NS}$-$\mathfrak{a}$-cofiniteness in the cases $\mathrm{dim}R=d$ or $\mathrm{dim}R/\mathfrak{a}=d-1$, where $d$ is a positive integer. We also study $\mathcal{NS}$-$\mathfrak{a}$-cofiniteness of local cohomology modules and the modules $\mathrm{Ext}^i_R(N,M)$ and $\mathrm{Tor}_i^R(N,M)$.\\
\vbox to 0.3cm{}

\leftskip0truemm \rightskip0truemm
\bigskip
\section* { \bf Introduction and Preliminaries}
 Throughout this paper, $R$ is a commutative noetherian ring
with identity, $\mathfrak{a}$ is a proper ideal of $R$ and $\mathcal{S}$ is a Serre
subcategory of $R$-modules, that is, $\mathcal{S}$ is closed
under taking submodules, quotients and extensions. Alipour and Sazeedeh \cite{AS} introduced the cofiniteness with respect to $\mathcal{S}$ and $\mathfrak{a}$.
An $R$-module $M$ is said to be \emph{$\mathcal{S}$-$\mathfrak{a}$-cofinite} if $\mathrm{Supp}_RM\subseteq\mathrm{V}(\mathfrak{a})$ and
$\mathrm{Ext}^i_R(R/\mathfrak{a},M)\in \mathcal{S}$ for all $i\geq 0$.

Let $\mathcal{N}$ be the subcategory of finitely generated $R$-modules.
The extension subcategory induced by $\mathcal{N}$ and $\mathcal{S}$ is denoted by $\mathcal{NS}$, consisting of those $R$-modules
$M$ for which there exist an exact sequence $0\rightarrow N\rightarrow M\rightarrow S\rightarrow 0$ such that $N\in\mathcal{N}$ and
$S\in\mathcal{S}$. It follows from \cite[Corollary 3.3]{Y} that $\mathcal{NS}$ is Serre. When $\mathcal{S}=0$, an $\mathcal{NS}$-$\mathfrak{a}$-cofinite module was known
as classical $\mathfrak{a}$-cofinite module, defined for the first time by Hartshorne \cite{H}, giving a negative answer to a question of \cite[Expos XIII, Conjecture 1.1]{G}, studied by numerous authors \cite{BN,BN1,BNS,LM,LM1,NS}. When $\mathcal{S}=\mathcal{A}$ the subcategory of artinian
modules, they are $\mathfrak{a}$-cominimax modules studies in \cite{BN,Z} and when $\mathcal{S}=\mathcal{F}$ the subcategory of all modules of finite support, they are $\mathfrak{a}$-weakly cofinite modules studies in \cite{Ba,DM}.
Recall that \emph{$\mathcal{S}$ satisfies the
condition $C_\mathfrak{a}$} if for every $R$-module $M$, the following implication holds.
\begin{center}$C_\mathfrak{a}$: If $\Gamma_\mathfrak{a}(M)=M$ and $(0:_M\mathfrak{a})$ is in $\mathcal{S}$, then $M$ is in $\mathcal{S}$.\end{center} By \cite[Lemma 2.2]{AM}, the following Serre subcategories satisfy the condition $C_\mathfrak{a}$.
 The class of zero modules;
 The class of artinian $R$-modules;
The class of artinian $\mathfrak{a}$-cofinite $R$-modules;
The class of $R$-modules with finite support; The class of $R$-modules with finite Krull dimension.
In this paper, we always assume that $\mathcal{S}$ satisfies the condition $C_\mathfrak{a}$.

The support of the Serre subcategory $\mathcal{S}$ is denoted by $\mathrm{Supp}\mathcal{S}$ which is
\begin{center}$\mathrm{Supp}\mathcal{S}=\bigcup_{M\in\mathcal{S}}\mathrm{Supp}_RM=
\{\mathfrak{p}\in\mathrm{Spec}R\hspace{0.03cm}|\hspace{0.03cm}R/\mathfrak{p}\in\mathcal{S}\}$.\end{center}For an $R$-module $M$, we denote by $\mathrm{Max}M$ the set of maximal ideal in $\mathrm{Supp}_RM$.
Assume that $\mathcal{S}$ satisfies the
condition $C_\mathfrak{a}$. Alipour and Sazeedeh \cite{AS,S} extended the fundamental results about $\mathfrak{a}$-cofinite modules at
small dimensions to $\mathcal{NS}$-$\mathfrak{a}$-cofinite modules. They showed that if $M$
is an $\mathcal{NS}$-$\mathfrak{a}$-cofinite $R$-module of dimension $\leq 1$ with $\mathrm{Max}M\subseteq\mathrm{Supp}\mathcal{S}$ and $N$ is a finitely generated $R$-module, then $\mathrm{Ext}^i_R(N,M)$ is $\mathcal{NS}$-$\mathfrak{a}$-cofinite for each $i\geq 0$ (see \cite[Theorem 2.7]{AS}); if $\mathrm{dim}R/\mathfrak{a}=1$ and $\mathrm{Max}M\subseteq\mathrm{Supp}\mathcal{S}$ then $M$
is $\mathcal{NS}$-$\mathfrak{a}$-cofinite if and only if $\mathrm{Supp}_RM\subseteq\mathrm{V}(\mathfrak{a})$ and $\mathrm{Ext}^i_R(R/\mathfrak{a},M)\in\mathcal{NS}$ for $i=0,1$ (see \cite[Theorem 3.2]{AS}); if $R$ is a local ring with $\mathrm{dim}R/\mathfrak{a}=2$ and satisfies some further conditions, then an $R$-module $M$
is $\mathcal{NS}$-$\mathfrak{a}$-cofinite if and only if $\mathrm{Supp}_RM\subseteq\mathrm{V}(\mathfrak{a})$ and $\mathrm{Ext}^i_R(R/\mathfrak{a},M)\in\mathcal{NS}$ for $i=0,1,2$ (see \cite[Corollary 2.11]{S}). They also investigated $\mathcal{NS}$-$\mathfrak{a}$-cofiniteness of local cohomology modules (see \cite[Theorem 2.13]{S}).

 The first aim of this paper is to improve Alipour and Sazeedeh's results in \cite{AS}, that is to say, eliminate the hypothesis $\mathrm{Max}M\subseteq\mathrm{Supp}\mathcal{S}$ entirely. We show that

 \vspace{2mm} \noindent{\bf Theorem 1.}\label{Th1.4} {\it{Let $M$ be an $R$-module with $\mathrm{dim}_RM\leq1$. Then
$M$ is $\mathcal{NS}$-$\mathfrak{a}$-cofinite if and only if $\mathrm{Supp}_RM\subseteq\mathrm{V}(\mathfrak{a})$ and  $\mathrm{Ext}^i_R(R/\mathfrak{a},M)\in\mathcal{NS}$ for $i=0,1$ {\rm (see Theorem \ref{lem:2.3}).}}}
\vspace{2mm}

 The second aim of this paper is to extend the results about $\mathfrak{a}$-cofiniteness in the cases $\mathrm{dim}R=d\geq1$ or $\mathrm{dim}R/\mathfrak{a}=d-1$ to $\mathcal{NS}$-$\mathfrak{a}$-cofiniteness, and improve Sazeedeh's some results in \cite{S}. More precisely, we show that

\vspace{2mm} \noindent{\bf Theorem 2.}\label{Th1.4} {\it{Let $\mathfrak{a}$ be an ideal of $R$ such that either $\mathrm{dim}R/\mathfrak{a}=d-1$ or $\mathrm{dim}R=d$. Then
an $\mathfrak{a}$-torsion $R$-module $M$ is $\mathcal{NS}$-$\mathfrak{a}$-cofinite if and only if $\mathrm{Ext}^i_R(R/\mathfrak{a},M)\in\mathcal{NS}$ for $i\leq d-1$ {\rm (see Theorem \ref{lem:2.10} and Corollary \ref{lem:2.6}).}}}
\vspace{2mm}

\noindent{\bf Theorem 3.}\label{Th1.4} {\it{Let $M$ be an $\mathcal{NS}$-$\mathfrak{a}$-cofinite $R$-module and $N$ a finitely generated $R$-module with $\mathrm{dim}_RN\leq2$. Then the $R$-modules $\mathrm{Ext}^i_R(N,M)$ and $\mathrm{Tor}_i^R(N,M)$ are $\mathcal{NS}$-$\mathfrak{a}$-cofinite for all $i\geq0$ {\rm (see Theorem \ref{lem:3.7}).}}}
\vspace{2mm}

As applications of these results, we show that if either $\mathrm{dim}R\leq2$ or $\mathrm{dim}R/\mathfrak{a}\leq 1$ then the subcategory $\mathcal{NS}(R,\mathfrak{a})_{cof}=\{M\in\textrm{Mod-}R\hspace{0.03cm}|\hspace{0.03cm}M\ \textrm{is}\ \mathcal{NS}\textrm{-}\mathfrak{a}\textrm{-cofinite}\}$ is abelian, and some results about $\mathcal{NS}$-$\mathfrak{a}$-cofiniteness of local cohomology modules are given.

Next we recall some notions which we will need later.

We write $\mathrm{Spec}R$ for the set of
prime ideals of $R$ and $\mathrm{Max}R$ for the set of
 maximal ideals of $R$. For an ideal $\mathfrak{a}$ in $R$, we set
\begin{center}$\mathrm{V}(\mathfrak{a})=\{\mathfrak{p}\in\textrm{Spec}R\hspace{0.03cm}|\hspace{0.03cm}\mathfrak{a}\subseteq\mathfrak{p}\}$.
\end{center}

Let $M$ be an $R$-module. The \emph{associated prime} of $M$, denoted by $\mathrm{Ass}_RM$, is
the set of prime ideals $\mathfrak{p}$ of $R$
such that there exists a cyclic submodule $N$ of $M$ with $\mathfrak{p}=\mathrm{Ann}_RN$. The set of prime ideals $\mathfrak{p}$ such that there exists a cyclic submodule $N$
of $M$ with $\mathfrak{p}\supseteq\mathrm{Ann}_RN$ is well-known to be the \emph{support} of $M$, denoted by $\mathrm{Supp}_RM$, which is equal to the set
 \begin{center}$\{\mathfrak{p}\in\mathrm{Spec}R\hspace{0.03cm}|\hspace{0.03cm}
M_\mathfrak{p}\neq0\}$.\end{center}
A prime ideal $\mathfrak{p}$ is said to be an \emph{attached prime} of $M$ if $\mathfrak{p}=\mathrm{Ann}_{R}(M/L)$ for some submodule $L$ of $M$. The set of attached primes of $M$ is denoted by $\mathrm{Att}_{R}M$. If $M$ is artinian, then $M$ admits a minimal secondary representation
$M=M_{1}+\cdots+M_{r}$ so that $M_{i}$ is $\mathfrak{p}_{i}$-secondary for $i=1,\cdots,r$. In this case, $\mathrm{Att}_{R}M=\{\mathfrak{p}_{1},\cdots,\mathfrak{p}_{r}\}$.

The \emph{arithmetic
rank of $\mathfrak{a}$}, denoted by
$\mathrm{ara}(\mathfrak{a})$, is the least number of elements of $R$ required to generate an ideal which has
the same radical as $\mathfrak{a}$, i.e.,
\begin{center}$\mathrm{ara}(\mathfrak{a})=\mathrm{min}\{n\geq0\hspace{0.03cm}|\hspace{0.03cm}\exists\ a_1,\cdots,a_n\in R\ \textrm{with}\ \mathrm{Rad}(a_1,\cdots,a_n)=\mathrm{Rad}(\mathfrak{a})\}$.\end{center}For an $R$-module $M$, the \emph{arithmetic rank of $\mathfrak{a}$ with respect
to $M$}, denoted by $\mathrm{ara}_M(\mathfrak{a})$, is defined by the arithmetic rank of the ideal $\mathfrak{a}+
\mathrm{Ann}_RM/\mathrm{Ann}_RM$ in the ring $R/\mathrm{Ann}_RM$.

The \emph{$i$th local cohomology} of an $R$-module $M$ with respect to $\mathfrak{a}$ is
\begin{center}$\mathrm{H}^i_\mathfrak{a}(M):=\underrightarrow{\textrm{lim}}_{t>0}\mathrm{Ext}^i_R(R/\mathfrak{a}^t,M)$.\end{center}
The reader can refer to \cite{BS} for more details about local cohomology. The module $M$ is called \emph{$\mathfrak{a}$-torsion} if $\Gamma_\mathfrak{a}(M):=\mathrm{H}^0_\mathfrak{a}(M)=M$, or equivalently, $\mathrm{Supp}_RM\subseteq\mathrm{V}(\mathfrak{a})$.

For an arbitrary $R$-module $M$, set
\begin{center}$\mathrm{cd}(\mathfrak{a},M)=\mathrm{sup}\{n\in\mathbb{Z}\hspace{0.03cm}|\hspace{0.03cm}\mathrm{H}_\mathfrak{a}^n(M)\neq0\}$.\end{center}
The \emph{cohomological dimension of $\mathfrak{a}$} is
\begin{center}$\mathrm{cd}(\mathfrak{a},R)=\mathrm{sup}\{\mathrm{cd}(\mathfrak{a},M)\hspace{0.03cm}|\hspace{0.03cm}M\ \textrm{is\ an}\ R\textrm{-module}\}$.\end{center}

\bigskip
\section{\bf Cofiniteness with respect to extension subcategories}
Let $d$ be a positive integer such that either $\mathrm{dim}R/\mathfrak{a}=d-1$ or $\mathrm{dim}R=d$. It is shown that an $R$-module
$M$ is $\mathcal{NS}$-$\mathfrak{a}$-cofinite if and only if $\mathrm{Supp}_RM\subseteq\mathrm{V}(\mathfrak{a})$ and $\mathrm{Ext}^i_R(R/\mathfrak{a},M)\in\mathcal{NS}$ for $i=0,\cdots,d-1$. Moreover, we show that the subcategory $\mathcal{NS}(R,\mathfrak{a})_{cof}$ is abelian in the cases $\mathrm{dim}R\leq2$ and $\mathrm{dim}R/\mathfrak{a}\leq 1$.

\begin{lem}\label{lem:2.2}{\it{Let $M$ be
an $R$-module such that $(0:_M\mathfrak{a})\in\mathcal{NS}$. Then $(0:_M\mathfrak{a}^n)\in\mathcal{NS}$ for all $n\geq1$.}}
\end{lem}
\begin{proof} This follows from the proof of \cite[Theorem 2.15]{S}.
\end{proof}

The following lemma is used at several places of this paper.

\begin{lem}\label{lem:2.1}{\it{Let $M$ be an $R$-module of zero dimension. Then $M$ is $\mathcal{NS}$-$\mathfrak{a}$-cofinite if and
only if $\mathrm{Supp}_RM\subseteq\mathrm{V}(\mathfrak{a})$ and $\mathrm{Hom}_R(R/\mathfrak{a},M)\in\mathcal{NS}$.}}
\end{lem}
\begin{proof} `Only if' part is trivial.

`If' part. By assumption, there exists a short exact sequence \begin{center}$0\rightarrow N\rightarrow\mathrm{Hom}_R(R/\mathfrak{a},M)\rightarrow S\rightarrow0$\end{center}with $N\in\mathcal{N}$ and $S\in\mathcal{S}$. If $\mathrm{Ass}_RM\subseteq\mathrm{Supp}\mathcal{S}$, then $\mathrm{Ass}_RN\subseteq\mathrm{Supp}\mathcal{S}$ and so $N\in\mathcal{S}$ by a finite filtration of $N$. Thus $\mathrm{Hom}_R(R/\mathfrak{a},M)\in\mathcal{S}$. Since $\mathcal{S}$ satisfies the condition $C_\mathfrak{a}$, one has $M\in\mathcal{S}$. Hence \cite[Lemma 2.1]{AT} implies that
$\mathrm{Ext}^i_R(R/\mathfrak{a},M)\in\mathcal{NS}$ for all $i\geq0$. Now assume that $\mathrm{Ass}_RM\nsubseteq\mathrm{Supp}\mathcal{S}$, and let $\Phi=\{\mathfrak{p}\hspace{0.03cm}|\hspace{0.03cm}\mathfrak{p}\in\mathrm{Ass}_RM\cap\mathrm{Supp}\mathcal{S}\}$. By \cite[Ch.IV, Section 1.2, Proposition 4]{B}, there is a submodule $K$ of $M$ such that $\mathrm{Ass}_RK=\mathrm{Ass}_RM\backslash\Phi$ and $\mathrm{Ass}_RM/K=\Phi\subseteq\mathrm{Supp}\mathcal{S}$.
As $\mathrm{Supp}_RK\cap\mathrm{Supp}\mathcal{S}=\emptyset$ and $\mathrm{Hom}_R(R/\mathfrak{a},K)\in\mathcal{NS}$, it follows that $\mathrm{Hom}_R(R/\mathfrak{a},K)$ has finite length. So $K$ is artinian $\mathfrak{a}$-cofinite by \cite[Proposition 4.1]{LM} and then $\mathrm{Ext}^i_R(R/\mathfrak{a},K)\in\mathcal{N}$ for all $i\geq0$. Hence the exact sequence $0\rightarrow K\rightarrow M\rightarrow M/K\rightarrow0$ implies that $\mathrm{Hom}_R(R/\mathfrak{a},M/K)\in\mathcal{NS}$. Since $\mathrm{Ass}_RM/K\subseteq\mathrm{Supp}\mathcal{S}$, by the preceding proof, $M/K\in\mathcal{S}$. Hence the above sequence yields that
$\mathrm{Ext}^i_R(R/\mathfrak{a},M)\in\mathcal{NS}$ for all $i\geq0$.
\end{proof}

We now present the first main theorem of this section, which eliminates the hypothesis $\mathrm{Max}M\subseteq\mathrm{Supp}\mathcal{S}$ in \cite[Theorem 3.2]{AS}.

\begin{thm}\label{lem:2.3}{\it{Let $M$ be an $R$-module with $\mathrm{dim}_RM\leq1$. Then
$M$ is $\mathcal{NS}$-$\mathfrak{a}$-cofinite if and only if $\mathrm{Supp}_RM\subseteq\mathrm{V}(\mathfrak{a})$ and  $\mathrm{Ext}^i_R(R/\mathfrak{a},M)\in\mathcal{NS}$ for $i=0,1$.}}
\end{thm}
\begin{proof} `Only if' part is obvious.

`If' part. By Lemma \ref{lem:2.1} we may assume $\mathrm{dim}_RM=1$, and let $t=\mathrm{ara}_M(\mathfrak{a})$. If $t=0$, then $M=(0:_M\mathfrak{a}^n)$ for
some $n\geq1$, and so the assertion follows by Lemma \ref{lem:2.2}. Next assume that $t>0$.
Let $\Phi=\{\mathfrak{p}\in\mathrm{Ass}_RM\cap\mathrm{Supp}\mathcal{S}\hspace{0.03cm}|\hspace{0.03cm}\mathrm{dim}R/\mathfrak{p}=1\}$. Then there is a submodule $K$ of $M$ so that $\mathrm{Ass}_RK=\Phi$ and $\mathrm{Ass}_RM/K=\mathrm{Ass}_RM\backslash\Phi$ by \cite[Ch.IV, Section 1.2, Proposition 4]{B}. As $\mathrm{Hom}_R(R/\mathfrak{a},K)\in\mathcal{NS}$, there exists an exact sequence
\begin{center}
$0\rightarrow N'\rightarrow\mathrm{Hom}_R(R/\mathfrak{a},K)\rightarrow S'\rightarrow0$
\end{center}with $N'\in\mathcal{N}$ and $S'\in\mathcal{S}$. Note that $\mathrm{Ass}_RK\subseteq\mathrm{Supp}\mathcal{S}$, so a finite filtration of $N'$ forces that $N'\in\mathcal{S}$, and hence $\mathrm{Hom}_R(R/\mathfrak{a},K)\in\mathcal{S}$. As $\mathcal{S}$ satisfies the condition $C_\mathfrak{a}$, one has $K\in\mathcal{S}$ and then $\mathrm{Ext}^i_R(R/\mathfrak{a},K)\in\mathcal{S}$ for all $i\geq0$. Replacing $M$ by $M/K$ we may assume that every $\mathfrak{p}\in\mathrm{Ass}_RM$ with $\mathrm{dim}R/\mathfrak{p}=1$ is not in $\mathrm{Supp}\mathcal{S}$. Let $\Phi=\{\mathfrak{p}\in\mathrm{Ass}_RM\hspace{0.03cm}|\hspace{0.03cm}\mathrm{dim}R/\mathfrak{p}=1\}$. There is a submodule $L$ of $M$ so that $\mathrm{Ass}_RL=\mathrm{Ass}_RM\backslash\Phi$ and $\mathrm{Ass}_RM/L=\Phi$ by \cite[Ch.IV, Section 1.2, Proposition 4]{B}. Since $\mathrm{Hom}_R(R/\mathfrak{a},L)\in\mathcal{NS}$ and $\mathrm{dim}_RL=0$, it follows from Lemma \ref{lem:2.1} that $L$ is $\mathcal{NS}$-$\mathfrak{a}$-cofinite and hence $\mathrm{Ext}^i_R(R/\mathfrak{a},M/L)\in\mathcal{NS}$ for $i=0,1$. Replacing $M$ by $M/L$ we may further assume that $\mathrm{Ass}_RM=\{\mathfrak{p}\in\mathrm{Supp}_RM\hspace{0.03cm}|\hspace{0.03cm}\mathrm{dim}R/\mathfrak{p}=1\}$.
Since $\mathrm{Hom}_R(R/\mathfrak{a},M)\in\mathcal{NS}$, there exists a short exact sequence
\begin{center}
$0\rightarrow N\rightarrow\mathrm{Hom}_R(R/\mathfrak{a},M)\rightarrow S\rightarrow0$
\end{center}with $N\in\mathcal{N}$ and $S\in\mathcal{S}$, which implies that the set
$\mathrm{Ass}_RM$ is finite. Also for each $\mathfrak{p}\in \mathrm{Ass}_RM$, the $R_\mathfrak{p}$-module $\mathrm{Hom}_{R_\mathfrak{p}}(R_\mathfrak{p}/\mathfrak{a}R_\mathfrak{p},M_\mathfrak{p})$ is finitely generated and $M_\mathfrak{p}$ is $\mathfrak{a}R_\mathfrak{p}$-torsion with $\mathrm{Supp}_{R_\mathfrak{p}}M_\mathfrak{p}\subseteq\mathrm{V}(\mathfrak{p}R_\mathfrak{p})$, it follows from \cite[Proposition 4.1]{LM}
that the $R_\mathfrak{p}$-module $M_\mathfrak{p}$ is artinian $\mathfrak{a}R_\mathfrak{p}$-cofinite. Let
$\mathrm{Ass}_RM=\{\mathfrak{p}_1,\cdots\mathfrak{,p}_n\}$.
It follows from \cite[Lemma 2.5]{BN} that $\mathrm{V}(\mathfrak{a}R_{\mathfrak{p}_j})\cap\mathrm{Att}_{R_{\mathfrak{p}_j}}M_{\mathfrak{p}_j}
\subseteq\mathrm{V}(\mathfrak{p}_jR_{\mathfrak{p}_j})$ for $j=1,\cdots,n$. Set \begin{center}
$\mathrm{U}=\bigcup_{j=1}^n\{\mathfrak{q}\in\mathrm{Spec}R\hspace{0.03cm}|\hspace{0.03cm}\mathfrak{q}R_{\mathfrak{p}_j}
\in\mathrm{Att}_{R_{\mathfrak{p}_j}}M_{\mathfrak{p}_j}\}$.
\end{center}Then $\mathrm{U}\cap\mathrm{V}(\mathfrak{a})\subseteq \mathrm{Ass}_RM$. On the other hand, for each $\mathfrak{q}\in\mathrm{U}$ we have $\mathfrak{q}R_{\mathfrak{p}_j}
\in\mathrm{Att}_{R_{\mathfrak{p}_j}}M_{\mathfrak{p}_j}$ for some $1\leq j\leq n$. Thus
\begin{center}
$(\mathrm{Ann}_RM)R_{\mathfrak{p}_j}\subseteq\mathrm{Ann}_{R_{\mathfrak{p}_j}}M_{\mathfrak{p}_j}\subseteq \mathfrak{q}R_{\mathfrak{p}_j}$,
\end{center}and so $\mathrm{Ann}_RM\subseteq\mathfrak{q}$. Since $t=\mathrm{ara}_M(\mathfrak{a})\geq 1$, there exist $y_1,\cdots,y_t\in\mathfrak{a}$
such that\begin{center}
$\mathrm{Rad}(\mathfrak{a}+\mathrm{Ann}_RM)=\mathrm{Rad}((y_1,\cdots,y_t)+\mathrm{Ann}_RM)$.
\end{center}As $\mathfrak{a}\nsubseteq(\bigcup_{\mathfrak{q}\in\mathrm{U}\backslash\mathrm{V}(\mathfrak{a})}\mathfrak{q})\cup
(\bigcup_{\mathfrak{p}\in\mathrm{Ass}_RM}\mathfrak{p})$, we have $(y_1,\cdots,y_t)\nsubseteq(\bigcup_{\mathfrak{q}\in\mathrm{U}\backslash\mathrm{V}(\mathfrak{a})}\mathfrak{q})\cup
(\bigcup_{\mathfrak{p}\in\mathrm{Ass}_RM}\mathfrak{p})$.
Hence \cite[Ex.16.8]{M}
provides an element $a_1\in(y_2,\cdots,y_t)$ so that $y_1+a_1\not\in(\bigcup_{\mathfrak{q}\in\mathrm{U}\backslash\mathrm{V}(\mathfrak{a})}\mathfrak{q})\cup
(\bigcup_{\mathfrak{p}\in\mathrm{Ass}_RM}\mathfrak{p})$. Set $x=y_1+a_1$. Then $x\in\mathfrak{a}$ and there is an exact sequence $0\rightarrow M\stackrel{x}\rightarrow M\rightarrow M/xM\rightarrow0$.
By assumption, $\mathrm{Hom}_{R}(R/\mathfrak{a},M/xM)\in\mathcal{NS}$, it follows from Lemma \ref{lem:2.1} that $M/xM$ is $\mathcal{NS}$-$\mathfrak{a}$-cofinite since $\mathrm{dim}_RM/xM=0$.
Therefore, by \cite[Lemma 2.2]{AS}, one has $M$ is $\mathcal{NS}$-$\mathfrak{a}$-cofinite, as desired.
\end{proof}

The following corollary generalizes \cite[Theorem 2.3]{LM1} and \cite[Theorem 3.2]{AS}.

\begin{cor}\label{lem:2.4}{\it{If $\mathrm{dim}R/\mathfrak{a}\leq 1$, then an $\mathfrak{a}$-torsion $R$-module $M$ is $\mathcal{NS}$-$\mathfrak{a}$-cofinite if and only if $\mathrm{Ext}^i_R(R/\mathfrak{a},M)\in\mathcal{NS}$ for $i=0,1$.}}
\end{cor}

An $R$-module $M$ is said to be \emph{weakly Laskerian} if the set $\mathrm{Ass}_RM/N$
is finite for each submodule $N$ of $M$.

\begin{cor}\label{lem:2.4'}{\it{Let $M$ be a weakly Laskerian $R$-module such that $\mathrm{Ext}^i_R(R/\mathfrak{a},M)\in\mathcal{NS}$ for $i=0,1$. Then $M$ is $\mathcal{NS}$-$\mathfrak{a}$-cofinite.}}
\end{cor}
\begin{proof} As $M$ is weakly Laskerian, there is an exact sequence $0\rightarrow N\rightarrow M\rightarrow F\rightarrow0$ such that $N\in\mathcal{N}$ and $F\in\mathcal{F}$ by \cite[Theorem 3.3]{Ba}. Note that $\mathrm{dim}_RF\leq 1$ and
 $\mathrm{Ext}^i_R(R/\mathfrak{a},F)\in\mathcal{NS}$ for $i=0,1$ by assumption, it follows from Theorem \ref{lem:2.3} that $F$ is $\mathcal{NS}$-$\mathfrak{a}$-cofinite, and then $M$ is $\mathcal{NS}$-$\mathfrak{a}$-cofinite.
\end{proof}

The next is the second main theorem of this section, which is a nice generalization of \cite[Theorem 2.3]{LM1} and \cite[Theorem 3.5]{BNS}.

\begin{thm}\label{lem:2.10}{\it{Assume that $\mathrm{dim}R/\mathfrak{a}=d\geq1$. Then
an $R$-module $M$ is $\mathcal{NS}$-$\mathfrak{a}$-cofinite if and only if $\mathrm{Supp}_RM\subseteq\mathrm{V}(\mathfrak{a})$ and $\mathrm{Ext}^i_R(R/\mathfrak{a},M)\in\mathcal{NS}$ for all $i\leq d$.}}
\end{thm}
\begin{proof}`Only if' part is trivial.

`If' part. We proceed by induction on $d$. If $d=1$ then the assertion follows by Corollary \ref{lem:2.4}. Suppose, inductively, $d>1$ and the result has been proved for smaller values of
$d$. If $\mathfrak{a}$ is nilpotent, say $\mathfrak{a}^n=0$ for some integer $n$, then $M=(0:_M\mathfrak{a}^n)\in\mathcal{NS}$ by Lemma \ref{lem:2.2} as $(0:_M\mathfrak{a})\in\mathcal{NS}$ and so $M$ is $\mathcal{NS}$-$\mathfrak{a}$-cofinite. Now assume that $\mathfrak{a}$ is not nilpotent. We can choose a positive
integer $n$ such that $(0:_R\mathfrak{a}^n)=\Gamma_\mathfrak{a}(R)$. Put $\overline{R}=R/\Gamma_\mathfrak{a}(R)$ and $\overline{M}=M/(0:_M\mathfrak{a}^n)$ which
is an $\overline{R}$-module. Taking $\overline{\mathfrak{a}}$ as the image of $\mathfrak{a}$ in $\overline{R}$, we have $\Gamma_{\overline{\mathfrak{a}}}(\overline{R})=0$. Thus $\overline{\mathfrak{a}}$ contains
an $\overline{R}$-regular element so that $\mathrm{dim}R/\mathfrak{a}+\Gamma_\mathfrak{a}(R)\leq d-1$. Note that $\mathrm{Supp}_R(R/\mathfrak{a}+\Gamma_\mathfrak{a}(R))\subseteq\mathrm{V}(\mathfrak{a})$, by the assumption and \cite[Lemma 2.1]{AS}, one has
$\mathrm{Ext}^i_R(R/\mathfrak{a}+\Gamma_\mathfrak{a}(R),M)\in\mathcal{NS}$ for $i\leq d$. Also $(0:_M\mathfrak{a}^n)\in\mathcal{NS}$, and thus $\mathrm{Ext}^i_R(R/\mathfrak{a}+\Gamma_\mathfrak{a}(R),\overline{M})\in\mathcal{NS}$ for $i\leq d$. On
the other hand, it is clear that $\mathrm{Supp}_R\overline{M}\subseteq\mathrm{V}(\mathfrak{a}+\Gamma_\mathfrak{a}(R))$. By the inductive hypothesis, the $R$-module $\overline{M}$ is $\mathcal{NS}$-$\mathfrak{a}+\Gamma_\mathfrak{a}(R)$-cofinite, and then $\overline{M}$ is $\mathcal{NS}$-$\mathfrak{a}$-cofinite by the proof of \cite[Theorem 2.15]{S}.
Therefore, $(0:_M\mathfrak{a}^n)\in\mathcal{NS}$ forces that
$M$ is $\mathcal{NS}$-$\mathfrak{a}$-cofinite.
\end{proof}

The following corollary is a nice generalization of \cite[Corollaries 2.16 and 2.18]{S}.

\begin{cor}\label{lem:2.6}{\it{If $\mathrm{dim}R=d\geq1$, then an $\mathfrak{a}$-torsion $R$-module $M$ is $\mathcal{NS}$-$\mathfrak{a}$-cofinite if and only if $\mathrm{Ext}^i_R(R/\mathfrak{a},M)\in\mathcal{NS}$ for all $i\leq d-1$.}}
\end{cor}
\begin{proof} Let $\mathfrak{a}$ be an ideal of $R$ with $\mathrm{dim}R/\mathfrak{a}\leq d-1$. It follows from Theorem \ref{lem:2.10} that $M$ is $\mathcal{NS}$-$\mathfrak{a}$-cofinite if and only if $\mathrm{Ext}^i_R(R/\mathfrak{a},M)\in\mathcal{NS}$ for all $i\leq d-1$. Hence \cite[Theorem 2.15]{S} yields the desired statement.
\end{proof}

 The next result eliminates the hypothesis $\mathrm{Max}M\subseteq\mathrm{Supp}\mathcal{S}$ in \cite[Theorem 3.4]{AS}.

\begin{cor}\label{lem:2.7}{\it{$(1)$ Let $\mathcal{NS}^1(R,\mathfrak{a})_{cof}$ denote
the category of $\mathcal{NS}$-$\mathfrak{a}$-cofinite $R$-modules $M$ with $\mathrm{dim}_RM \leq1$. Then $\mathcal{NS}^1(R,\mathfrak{a})_{cof}$ is abelian.

(2) If either $\mathrm{dim}R\leq2$ or $\mathrm{dim}R/\mathfrak{a}\leq 1$, then $\mathcal{NS}(R,\mathfrak{a})_{cof}$ is abelian.}}
\end{cor}
\begin{proof} We just prove (2) since the proof of (2) is similar.

Given an $R$-homomorphism $f:M\rightarrow N$ in $\mathcal{NS}^1(R,\mathfrak{a})_{cof}$, set $K=\mathrm{ker}f$, $I=\mathrm{im}f$ and $C=\mathrm{coker}f$. It
is easy to obtain that $\mathrm{Hom}_R(R/\mathfrak{a},K),\mathrm{Ext}^1_R(R/\mathfrak{a},K)\in\mathcal{NS}$ and hence
the module $K\in\mathcal{NS}^1(R,\mathfrak{a})_{cof}$ by Theorem \ref{lem:2.3}. This implies that $I\in\mathcal{NS}^1(R,\mathfrak{a})_{cof}$ and consequently $C\in\mathcal{NS}^1(R,\mathfrak{a})_{cof}$, as required.
\end{proof}

 The following corollary is a generalization \cite[Theorem 2.7]{AS}.

\begin{cor}\label{lem:2.7'}{\it{Let $M$ be an $\mathcal{NS}$-$\mathfrak{a}$-cofinite $R$-module with $\mathrm{dim}_RM \leq1$ and $N$ a finitely generated $R$-module.
Then the $R$-modules $\mathrm{Tor}^R_i(N,M)$ and $\mathrm{Ext}^i_R(N,M)$ are $\mathcal{NS}$-$\mathfrak{a}$-cofinite for all $i\geq 0$.}}
\end{cor}
\begin{proof} Since $N$ is finitely
generated, $N$
has a free resolution
\begin{center}$F^\bullet:\cdots\rightarrow F_n\rightarrow F_{n-1}\rightarrow\cdots\rightarrow F_1\rightarrow F_0\rightarrow0$,\end{center}
where all $F_i$ have finite ranks.
Then $\mathrm{Tor}^R_i(N,M)=\mathrm{H}_i(F^\bullet\otimes_RM)$ and $\mathrm{Ext}^i_R(N,M)=\mathrm{H}^i(\mathrm{Hom}_R(F^\bullet,M))$ are subquotients of a direct sum of finitely many copies of $M$.
Now, the assertion follows from Corollary \ref{lem:2.7}(1).
\end{proof}

The next result is a more general version of \cite[Theorem 2.8]{NS}.

\begin{cor}\label{lem:2.8}{\it{If either $\mathrm{dim}R=d\geq3$ or $\mathrm{dim}R/\mathfrak{a}=d-1$, then the subcategory $\mathcal{NS}(R,\mathfrak{a})_{cof}$ is abelian if and only if for any homomorphism $f:M\rightarrow N$ in $\mathcal{NS}(R,\mathfrak{a})_{cof}$ and $i\leq d-2$, $\mathrm{Ext}^i_R(R/\mathfrak{a},\mathrm{coker}f)\in\mathcal{NS}$.}}
\end{cor}
\begin{proof} `Only if' part is trivial.

`If' part. Since $\mathrm{Ext}^i_R(R/\mathfrak{a},\mathrm{coker}f)\in\mathcal{NS}$ for all $i\leq d-2$, we have $\mathrm{Ext}^i_R(R/\mathfrak{a},\mathrm{im}f)\in\mathcal{NS}$ for all $i\leq d-1$, and hence $\mathrm{im}f$ is $\mathcal{NS}$-$\mathfrak{a}$-cofinite by Theorem \ref{lem:2.10} and Corollary \ref{lem:2.6}. This implies that $\mathrm{ker}f$ and therefore $\mathrm{coker}f$ is $\mathcal{NS}$-$\mathfrak{a}$-cofinite, as desired.
\end{proof}

The following result is a generalization of \cite[Proposition 2.10]{NS}.

\begin{prop}\label{lem:3.11}{\it{Let $\mathfrak{a}$ and $\mathfrak{b}$ be two ideals of $R$ with $\mathfrak{b}\subseteq\mathfrak{a}$ and $M$ an $R$-module. If $n$ is a nonnegative integer such that $\mathrm{Ext}^i_R(R/\mathfrak{b},M)$ are $\mathcal{NS}$-$\mathfrak{a}$-cofinite for all $i\leq n$, then $\mathrm{Ext}^i_R(R/\mathfrak{a},M)\in\mathcal{NS}$ for all $i\leq n$.}}
\end{prop}
\begin{proof} Assume that $0\rightarrow M\rightarrow E^0\stackrel{d^0}\rightarrow E^1\stackrel{d^1}\rightarrow\cdots$ is an injective resolution of $M$. We get the exact sequences
$0\rightarrow M^i\rightarrow E^i\rightarrow M^{i+1}\rightarrow 0$ and isomorphisms
\begin{center}$\mathrm{Ext}^{i+1}_R(R/\mathfrak{a},M)\cong\mathrm{Ext}^{1}_R(R/\mathfrak{a},M^i),\
\mathrm{Ext}^{i+1}_R(R/\mathfrak{b},M)\cong\mathrm{Ext}^{1}_R(R/\mathfrak{b},M^i)$,\end{center}
where $M^i=\mathrm{ker}d^i$ for $i\geq0$.
Hence, for each $i\geq0$, there is an exact sequence
\begin{center}$0\rightarrow(0:_{M^i}\mathfrak{b})\rightarrow (0:_{E^i}\mathfrak{b}) \stackrel{f^i}\rightarrow(0:_{M^{i+1}}\mathfrak{b})\rightarrow\mathrm{Ext}^{i+1}_R(R/\mathfrak{b},M)\rightarrow 0$.\end{center}
We first show that $\mathrm{Ext}^s_{R/\mathfrak{b}}(R/\mathfrak{a},\mathrm{Ext}^{i}_R(R/\mathfrak{b},M))\in\mathcal{NS}$ for all $s\geq0$ and $0\leq i\leq n$. Consider the Grothendieck spectral sequences
 \begin{center}$E_2^{p,q}=\xymatrix@C=10pt@R=5pt{
 \mathrm{Ext}^{p}_{R/\mathfrak{b}}(\mathrm{Tor}_q^R(R/\mathfrak{b},R/\mathfrak{a}),\mathrm{Ext}^{i}_R(R/\mathfrak{b},M))\ar@{=>}[r]_{\ \ \ \ \ \ \  p}&
 \mathrm{Ext}^{p+q}_R(R/\mathfrak{a},\mathrm{Ext}^{i}_R(R/\mathfrak{b},M)).}$\end{center}
For $s=0$, we have $\mathrm{Hom}_{R/\mathfrak{b}}(R/\mathfrak{a},\mathrm{Ext}^{i}_R(R/\mathfrak{b},M))\cong
\mathrm{Hom}_{R}(R/\mathfrak{a},\mathrm{Ext}^{i}_R(R/\mathfrak{b},M))\in
\mathcal{NS}$ for $0\leq i\leq n$. Now, assume that $s> 0$ and the result has been proved for all values smaller than $s$. Then
$E_2^{p,0}=\mathrm{Ext}^{p}_{R/\mathfrak{b}}(R/\mathfrak{a},\mathrm{Ext}^{i}_R(R/\mathfrak{b},M))\in\mathcal{NS}$ for all $0\leq p<s$. Since $\mathrm{Supp}_{R/\mathfrak{b}}\mathrm{Tor}_q^R(R/\mathfrak{b},R/\mathfrak{a})\subseteq
\mathrm{Supp}_{R/\mathfrak{b}}R/\mathfrak{a}$, it follows from \cite[Lemma 2.4]{AS} that $E_2^{p,q}\in\mathcal{NS}$ for all $0\leq p<s$ and $q\geq0$. There exists a finite filtration
\begin{center}$0=\Phi^{s+1}H^s\subset\cdots\subset \Phi^{1}H^s\subset\Phi^{0}H^s\subset H^s:=\mathrm{Ext}^{s}_R(R/\mathfrak{a},\mathrm{Ext}^{i}_R(R/\mathfrak{b},M))$,\end{center}
such that $E_\infty^{s,0}\cong\Phi^{s}H^s/\Phi^{s+1}H^s=\Phi^{s}H^s$
is a submodule of $H^s\in\mathcal{NS}$, and so $E_\infty^{s,0}\in\mathcal{NS}$. For $r\geq2$, consider the differential
\begin{center}$E_r^{s-r,r-1}\xrightarrow{d_r^{s-r,r-1}}E_r^{s,0}
\xrightarrow{d_r^{s,0}}E_r^{s+r,-r+1}=0$.
\end{center} We have an exact sequence $E_r^{s-r,r-1}\rightarrow E_r^{s,0}\rightarrow E_{r+1}^{s,0}\rightarrow0$. As $E_r^{s,0}\cong E_\infty^{s,0}\in\mathcal{NS}$ for $r\gg0$, the sequence
implies that $E_2^{s,0}=\mathrm{Ext}^s_{R/\mathfrak{b}}(R/\mathfrak{a},\mathrm{Ext}^{i}_R(R/\mathfrak{b},M))\in\mathcal{NS}$ for all $s\geq 0$ and $0\leq i\leq n$.
Next consider the Grothendieck spectral sequences
 \begin{center}$E_2^{p,q}=\xymatrix@C=10pt@R=5pt{
 \mathrm{Ext}^{p}_{R/\mathfrak{b}}(R/\mathfrak{a},\mathrm{Ext}^{q}_R(R/\mathfrak{b},M))\ar@{=>}[r]_{\ \ \ \ \ \ \  p}&
\mathrm{Ext}^{p+q}_R(R/\mathfrak{a},M).}$\end{center}
For $0\leq i\leq n$, there exists a finite filtration \begin{center}$
 0=\Phi^{i+1}H^{i}\subseteq \Phi^{i}H^{i}\subseteq\cdots\subseteq \Phi^{1}H^{i}\subseteq \Phi^{0}H^{i}=H^{i}:=\mathrm{Ext}^{i}_R(R/\mathfrak{a},M)$,
\end{center}such that $\Phi^{p}H^{i}/\Phi^{p+1}H^{i}\cong E_\infty^{p,i-p}$ for $0\leq p\leq i$. As $E_\infty^{p,i-p}$ is a subquotient of $E_2^{p,i-p}$, a successive use of the exact sequence
\begin{center}$
 0\rightarrow \Phi^{p+1}H^i\rightarrow \Phi^{p}H^i\rightarrow \Phi^{p}H^i/\Phi^{p+1}H^i\rightarrow0$
\end{center} implies that $\mathrm{Ext}^i_R(R/\mathfrak{a},M)\in\mathcal{NS}$ for all $i\leq n$.
\end{proof}

\begin{cor}\label{lem:3.11'}{\it{Let $\mathfrak{a}$ and $\mathfrak{b}$ be two ideals of $R$ with $\mathfrak{b}\subseteq\mathfrak{a}$ and $M$ an $\mathfrak{a}$-torsion $R$-module.

$(1)$ If $\mathrm{Ext}^i_R(R/\mathfrak{b},M)$ is $\mathcal{NS}$-$\mathfrak{a}$-cofinite for each $i\geq0$, then $M$ is $\mathcal{NS}$-$\mathfrak{a}$-cofinite.

$(2)$ For a non-negative integer $d$, if $\mathrm{dim}R/\mathfrak{a}=d$ and $\mathrm{Ext}^i_R(R/\mathfrak{b},M)$ is $\mathcal{NS}$-$\mathfrak{a}$-cofinite for $0\leq i\leq d$, then $M$ is $\mathcal{NS}$-$\mathfrak{a}$-cofinite.}}
\end{cor}

\bigskip
\section{\bf $\mathcal{NS}$-$\mathfrak{a}$-cofiniteness of local cohomology modules}
This section, we study $\mathcal{NS}$-$\mathfrak{a}$-cofiniteness of local cohomology modules. The following result generalizes \cite[Theorem 3.3 and Proposition 3.4]{NS} and \cite[Theorem 3.5]{AS}.

\begin{thm}\label{lem:3.2}{\it{Let $M$ be an $R$-module and $n$ a non-negative integer. If either $\mathrm{dim}_RM\leq 1$ or $\mathrm{dim}R/\mathfrak{a}\leq 1$ or $\mathrm{dim}R\leq 2$, then $\mathrm{Ext}^i_R(R/\mathfrak{a},M)\in\mathcal{NS}$ for all $i\leq n+1$ if and only if $\mathrm{H}^i_\mathfrak{a}(M)$ is $\mathcal{NS}$-$\mathfrak{a}$-cofinite for all $i\leq n$ and $\mathrm{Hom}_R(R/\mathfrak{a},\mathrm{H}^{n+1}_\mathfrak{a}(M))\in\mathcal{NS}$.}}
\end{thm}
\begin{proof} `If' part follows from \cite[Theorem 2.1]{BA}.

`Only if' part.
 Set $s=1$ in \cite[Theorem 2.9]{BA}, it is enough to show that $\mathrm{H}^{i}_\mathfrak{a}(M)$ are $\mathcal{NS}$-$\mathfrak{a}$-cofinite
for all $i\leq n$. We prove by induction on $n$. If $n=0$ and $\mathrm{Ext}^i_R(R/\mathfrak{a},M)\in\mathcal{NS}$ for $i=0,1$, then $\mathrm{Hom}_R(R/\mathfrak{a},\Gamma_\mathfrak{a}(M)),\mathrm{Ext}^1_R(R/\mathfrak{a},\Gamma_\mathfrak{a}(M))\in\mathcal{NS}$, and so
$\Gamma_\mathfrak{a}(M)$ is $\mathcal{NS}$-$\mathfrak{a}$-cofinite by Theorem \ref{lem:2.3} and Corollaries \ref{lem:2.4} and \ref{lem:2.6} and $\mathrm{Hom}_R(R/\mathfrak{a},\mathrm{H}^{1}_\mathfrak{a}(M))\in\mathcal{NS}$ by \cite[Theorem 2.9]{BA}. Now, suppose that $n>0$ and the result has been proved for smaller values of
$n$. Then $\mathrm{H}^i_\mathfrak{a}(X)$ is $\mathcal{NS}$-$\mathfrak{a}$-cofinite for $i\leq n-1$ by the induction. Hence
\cite[Theorem 2.9]{BA} implies that $\mathrm{Ext}^i_R(R/\mathfrak{a},\mathrm{H}^{n}_\mathfrak{a}(M))\in\mathcal{NS}$ for $i=0,1$, and hence
 $\mathrm{H}^{n}_\mathfrak{a}(M)$ is $\mathcal{NS}$-$\mathfrak{a}$-cofinite by Corollaries \ref{lem:2.4} and \ref{lem:2.6} and $\mathrm{Hom}_R(R/\mathfrak{a},\mathrm{H}^{n+1}_\mathfrak{a}(M))\in\mathcal{NS}$ by \cite[Theorem 2.9]{BA}.
\end{proof}

\begin{cor}\label{lem:3.2'}{\it{Let $\mathfrak{a},\mathfrak{b}$ be two ideals of $R$ with $\mathfrak{b}\subseteq\mathfrak{a}$, $n$ a non-negative integer and $M$ be an $R$-module such that $\mathrm{H}^{i}_\mathfrak{b}(M)$ is $\mathcal{NS}$-$\mathfrak{a}$-cofinite for $i\leq n+1$. If either $\mathrm{dim}R/\mathfrak{a}\leq 1$ or $\mathrm{dim}R\leq 2$, then $\mathrm{H}^i_\mathfrak{a}(M)$ are $\mathcal{NS}$-$\mathfrak{a}$-cofinite for all $i\leq n$.}}
\end{cor}
\begin{proof} Consider the Grothendieck spectral sequence
\begin{center}$\xymatrix@C=10pt@R=5pt{
 E_2^{p,q}=\mathrm{Ext}^{p}_R(R/\mathfrak{a},\mathrm{H}^{q}_\mathfrak{b}(M))\ar@{=>}[r]_{\ \ \ \ \ \ p}&
 \mathrm{Ext}^{p+q}_R(R/\mathfrak{a},M).}$\end{center}For $0\leq i\leq n+1$, there
exists a finite filtration
\begin{center}$
 0=\Phi^{i+1}H^{i}\subseteq \Phi^{i}H^{i}\subseteq\cdots\subseteq \Phi^{1}H^{i}\subseteq \Phi^{0}H^{i}=H^{i}:=\mathrm{Ext}^{i}_R(R/\mathfrak{a},M)$,
\end{center}such that $\Phi^{p}H^{i}/\Phi^{p+1}H^{i}\cong E_\infty^{p,i-p}$ for $0\leq p\leq i$. As $E_\infty^{p,i-p}$ is a subquotient of $E_2^{p,i-p}$, a successive use of the exact sequence
\begin{center}$
 0\rightarrow \Phi^{p+1}H^i\rightarrow \Phi^{p}H^i\rightarrow \Phi^{p}H^i/\Phi^{p+1}H^i\rightarrow0$
\end{center} implies that $\mathrm{Ext}^{i}_R(R/\mathfrak{a},M)\in\mathcal{NS}$ for $i\leq n+1$, and hence, by Theorem \ref{lem:3.2}, $\mathrm{H}^i_\mathfrak{a}(M)$ are $\mathcal{NS}$-$\mathfrak{a}$-cofinite for all $i\leq n$.
\end{proof}

\begin{cor}\label{lem:3.22}{\it{Let $M$ be a weakly Laskerian $R$-module such that $\mathrm{Ext}^{i}_R(R/\mathfrak{a},M)\in\mathcal{NS}$ for $i=0,1$. If either $\mathrm{dim}R/\mathfrak{a}\leq 1$ or $\mathrm{dim}R\leq 2$, then $\mathrm{H}^i_\mathfrak{a}(M)$ is $\mathcal{NS}$-$\mathfrak{a}$-cofinite for every $i\geq 0$.}}
\end{cor}
\begin{proof} As $M$ is weakly Laskerian, there is an exact sequence $0\rightarrow N\rightarrow M\rightarrow F\rightarrow0$ so that $N\in\mathcal{N}$ and $F\in\mathcal{F}$ by \cite[Theorem 3.3]{Ba}. Then $F$ is $\mathcal{NS}$-$\mathfrak{a}$-cofinite by Theorem \ref{lem:2.3}. Hence the above sequence and Theorem \ref{lem:3.2} yield the desired statement.
\end{proof}

The next corollary is a more general version of \cite[Theorem 2.15]{BN} and \cite[Theorem 2.6]{Ma}.

\begin{cor}\label{lem:3.2''}{\it{Let $\mathfrak{a},\mathfrak{b}$ be two ideals of $R$ with $\mathfrak{b}\subseteq\mathfrak{a}$, $n$ a non-negative integer and $M$ be an $R$-module such that $\mathrm{Ext}^{i}_R(R/\mathfrak{b},M)\in\mathcal{NS}$ for $i\leq n+1$. If $\mathrm{dim}R/\mathfrak{a}=\mathrm{dim}R/\mathfrak{b}\leq 1$, then $\mathrm{H}^i_\mathfrak{a}(\mathrm{H}^j_\mathfrak{b}(M))$ is $\mathcal{NS}$-$\mathfrak{a}$-cofinite for all $i\geq0$ and $j\leq n$.}}
\end{cor}
\begin{proof} By Theorem \ref{lem:3.2}, one has $\mathrm{H}^j_\mathfrak{b}(M)$ are $\mathcal{NS}$-$\mathfrak{b}$-cofinite for all $j\leq n$, which implies that $\mathrm{Ext}^{i}_R(R/\mathfrak{a},\mathrm{H}^j_\mathfrak{b}(M))\in\mathcal{NS}$ for all $i$ and $j\leq n$. Hence $\mathrm{H}^i_\mathfrak{a}(\mathrm{H}^j_\mathfrak{b}((M))$ are $\mathcal{NS}$-$\mathfrak{a}$-cofinite for all $i\geq0$ and $j\leq n$ by Theorem \ref{lem:3.2} again.
\end{proof}

The next corollary is a generalization of \cite[Corollary 3.14 and Theorem 7.10]{LM} and \cite[Corllary 2.12]{LM1}.

\begin{cor}\label{lem:3.02'}{\it{If either $\mathrm{dim}R/\mathfrak{a}\leq 1$ or $\mathrm{dim}R\leq 2$ or $\mathrm{cd}(\mathfrak{a},R)\leq1$, then $\mathrm{H}^i_\mathfrak{a}(M)$ is $\mathcal{NS}$-$\mathfrak{a}$-cofinite for any $M\in\mathcal{NS}$ and every $i\geq0$.}}
\end{cor}
\begin{proof} This follows from Theorem \ref{lem:3.2} and \cite[Theorem 2.9]{BA}.
\end{proof}

\begin{cor}\label{lem:3.4}{\it{Let $M\neq 0$ be in $\mathcal{NS}$ such that
$\mathrm{dim}_RM/\mathfrak{a}M\leq1$. Then for each finitely generated $R$-module $N$, the $R$-modules
$\mathrm{Ext}^i_R(N,\mathrm{H}^j_\mathfrak{a}(M))$ are
$\mathcal{NS}$-$\mathfrak{a}$-cofinite for all $i,j\geq 0$.}}
\end{cor}
\begin{proof} As $\mathrm{Supp}_R\mathrm{H}^j_\mathfrak{a}(M)\subseteq\mathrm{Supp}_RM/\mathfrak{a}M$, it follows from Theorem \ref{lem:3.2} that
$\mathrm{H}^j_\mathfrak{a}(M)$ are
$\mathcal{NS}$-$\mathfrak{a}$-cofinite for all $j\geq 0$. Now the assertion follows from Corollary \ref{lem:2.7}.
\end{proof}

The following proposition is a more general version of \cite[Theorem 3.7]{BNS}.

\begin{prop}\label{lem:3.3}{\it{Let $n$ be a non-negative integer such that $\mathrm{Ext}^i_R(R/\mathfrak{a},M)\in\mathcal{NS}$ for all $i\leq n+1$. If either $\mathrm{dim}R=d\geq3$ or $\mathrm{dim}R/\mathfrak{a}=d-1$, then $\mathrm{H}^{i}_\mathfrak{a}(M)$ is $\mathcal{NS}$-$\mathfrak{a}$-cofinite for $i<n$ if and only if $\mathrm{Hom}_R(R/\mathfrak{a},\mathrm{H}^{i+d-3}_\mathfrak{a}(M)),\cdots,\mathrm{Ext}^{d-3}_R(R/\mathfrak{a},
 \mathrm{H}^{i}_\mathfrak{a}(M))\in\mathcal{NS}$ for $i\leq n$.}}
\end{prop}
\begin{proof} This follows from \cite[Theorem 2.9]{BA} and Theorem \ref{lem:2.10} and Corollary \ref{lem:2.6}.
\end{proof}

The next result is a generalization of \cite[Proposition 5.1]{LM}.

\begin{prop}\label{lem:3.9}{\it{Let $M\in\mathcal{NS}$ be an $R$-module of dimesnsion $d$. Then the top local cohomology module
$\mathrm{H}_\mathfrak{a}^d(M)$ is $\mathcal{NS}$-$\mathfrak{a}$-cofinite of zero dimension.}}
\end{prop}
\begin{proof} We use induction on $d$. This is clear if $d=0$.
So assume that $d>0$ and replacing $M$ with $M/\Gamma_\mathfrak{a}(M)$, we may assume that $\mathfrak{a}$ contains an $M$-regular element $x$. By induction, $\mathrm{H}_\mathfrak{a}^{d-1}(M/xM)$ is $\mathcal{NS}$-$\mathfrak{a}$-cofinite of zero dimension. Then  the exact
sequence\begin{center}$\mathrm{H}_\mathfrak{a}^{d-1}(M/xM)\rightarrow\mathrm{H}_\mathfrak{a}^d(M)\stackrel{x}
\rightarrow\mathrm{H}_\mathfrak{a}^d(M)\rightarrow0$\end{center}and Lemma \ref{lem:3.8} imply that
$(0:_{\mathrm{H}_\mathfrak{a}^d(M)}x)$ is $\mathcal{NS}$-$\mathfrak{a}$-cofinite of zero dimension. Thus, by \cite[Lemma 2.2]{AS}, $\mathrm{H}_\mathfrak{a}^d(M)$ is $\mathcal{NS}$-$\mathfrak{a}$-cofinite of zero dimension.
\end{proof}

An $R$-module $M$ is \emph{minimax} if there is a
finitely generated submodule $N$ of $M$, such that $M/N$ is artinian.

\begin{cor}\label{lem:3.3''}{\it{Let $M$ be a minimax $R$-module of dimension $d$. Then $\mathrm{H}_\mathfrak{a}^d(M)$ is artinian.}}
\end{cor}
\begin{proof} By Proposition \ref{lem:3.9}, there is an exact seuqnece \begin{center}$0\rightarrow N\rightarrow\mathrm{Hom}_R(R/\mathfrak{a},\mathrm{H}_\mathfrak{a}^d(M))\rightarrow A\rightarrow0$\end{center}with $N\in\mathcal{N}$ and $A\in\mathcal{A}$. But $\mathrm{dim}_R\mathrm{H}_\mathfrak{a}^d(M)=0$, it follows that $\mathrm{Hom}_R(R/\mathfrak{a},\mathrm{H}_\mathfrak{a}^d(M))$ artinian, and so $\mathrm{H}_\mathfrak{a}^d(M)$ is artinian.
\end{proof}

The following proposition is a generalization of \cite[Theorem 7.1.3]{BS}.

\begin{prop}\label{lem:3.10}{\it{If $R/\mathfrak{a}\in\mathcal{S}$, then $\mathrm{H}_\mathfrak{a}^i(M)\in\mathcal{S}$ for every $M\in\mathcal{NS}$ and all $i\geq0$.}}
\end{prop}
\begin{proof} We use induction on $i$. First since $M\in\mathcal{NS}$, there is an exact sequence $0\rightarrow N\rightarrow M\rightarrow S\rightarrow0$ with $N\in\mathcal{N}$ and $S\in\mathcal{S}$. Then $\mathrm{H}_\mathfrak{a}^0(N)=(0:_N\mathfrak{a}^n)$ for some $n\geq1$. Since $\mathrm{Ass}_R(0:_N\mathfrak{a}^n)\subseteq\mathrm{V}(\mathfrak{a})$, a finite filtration of $(0:_N\mathfrak{a}^n)$ forces that $\mathrm{H}_\mathfrak{a}^0(N)\in\mathcal{S}$. Also $\mathrm{H}_\mathfrak{a}^0(S)\in\mathcal{S}$, so $\mathrm{H}_\mathfrak{a}^0(M)\in\mathcal{S}$.
Now assume, inductively, that $i>0$ and that $\mathrm{H}_\mathfrak{a}^{i-1}(M')\in\mathcal{S}$ for all finitely generated $R$-modules $M'$. Since  $\mathrm{H}_\mathfrak{a}^i(M)\cong\mathrm{H}_\mathfrak{a}^i(M/\Gamma_\mathfrak{a}(M))$ for all $i>0$, we may assume that $\Gamma_\mathfrak{a}(M)=0$, and the
ideal $\mathfrak{a}$ contains an $M$-regular element $x$. Then the exact sequence $0\rightarrow M\stackrel{x}\rightarrow M\rightarrow M/xM\rightarrow0$ induces the following exact sequence\begin{center}$\mathrm{H}_\mathfrak{a}^{i-1}(M/xM)\rightarrow\mathrm{H}_\mathfrak{a}^i(M)\stackrel{x}
\rightarrow\mathrm{H}_\mathfrak{a}^i(M)$.\end{center}
By induction, $\mathrm{H}_\mathfrak{a}^{i-1}(M/xM)\in\mathcal{S}$, so
$(0:_{\mathrm{H}_\mathfrak{a}^i(M)}x)\in\mathcal{S}$, and then $(0:_{\mathrm{H}_\mathfrak{a}^i(M)}\mathfrak{a})\in\mathcal{S}$. As $\mathcal{S}$ satisfies the condition $C_\mathfrak{a}$, we have $\mathrm{H}_\mathfrak{a}^i(M)\in\mathcal{S}$. The inductive step is complete.
\end{proof}

\begin{cor}\label{lem:3.3'}{\it{$(1)$ Let $M$ be a minimax $R$-module. Then the $R$-module $\mathrm{H}_\mathfrak{m}^i(M)$ is artinian for every $i\geq0$ and $\mathfrak{m}\in\mathrm{Max}R$.

$(2)$ Let $R$ be a local ring and $M$ a weakly Laskerian $R$-module. If $\mathrm{dim}R/\mathfrak{a}\leq1$, then the set $\mathrm{Supp}_R\mathrm{H}^{i}_\mathfrak{a}(M)$ is finite for every $i\geq0$.}}
\end{cor}

\bigskip
\section{\bf $\mathcal{NS}$-$\mathfrak{a}$-cofiniteness for extension and torsion functors}
This section investigates $\mathcal{NS}$-$\mathfrak{a}$-cofiniteness of the $R$-modules $\mathrm{Ext}^i_R(N,M)$ and $\mathrm{Tor}_i^R(N,M)$. It is shown that
$\mathrm{Ext}^i_R(N,M)$ and $\mathrm{Tor}_i^R(N,M)$ are $\mathcal{NS}$-$\mathfrak{a}$-cofinite for all $i\geq0$ whenever $N$
is finitely generated with $\mathrm{dim}_RN\leq2$ and $M$ is $\mathcal{NS}$-$\mathfrak{a}$-cofinite.

\begin{lem}\label{lem:3.5'}{\it{Let $M$ be an $\mathcal{NS}$-$\mathfrak{a}$-cofinite $R$-module and $N$ a non-zero finite length $R$-module. Then $\mathrm{Ext}^i_R(N,M)$ and $\mathrm{Tor}_i^R(N,M)$ are in $\mathcal{NS}$ for all $i\geq0$.}}
\end{lem}
\begin{proof} This follows from \cite[Lemma 2.3]{AS} and \cite[Corollart 2.2.13]{F}.
\end{proof}

\begin{lem}\label{lem:3.6}{\it{Let $M$ be an $\mathcal{NS}$-$\mathfrak{a}$-cofinite $R$-module and $N$ a finitely generated $R$-module with $\mathrm{dim}_RN\leq1$. Then $\mathrm{Ext}^i_R(N,M)$ and $\mathrm{Tor}_i^R(N,M)$ are $\mathcal{NS}$-$\mathfrak{a}$-cofinite of zero dimension for every $i\geq0$.}}
\end{lem}
\begin{proof} By Lemma \ref{lem:3.5'}, we may assume
$\mathrm{dim}_RN=1$. It follows from \cite[Lemma 2.1]{AS} that
$\mathrm{Ext}^i_R(\Gamma_\mathfrak{a}(N),M)\in\mathcal{NS}$, and so
$\mathrm{Tor}_i^R(\Gamma_\mathfrak{a}(N),M)\in\mathcal{NS}$ for all $i\geq0$ by \cite[Corollary 2.2.13]{F}. The exact sequence
$0\rightarrow\Gamma_\mathfrak{a}(N)\rightarrow N\rightarrow N/\Gamma_\mathfrak{a}(N) \rightarrow 0$
induces the following two exact sequences \begin{center}$\mathrm{Ext}^{i-1}_R(\Gamma_\mathfrak{a}(N),M)\rightarrow\mathrm{Ext}^{i}_R(N/\Gamma_\mathfrak{a}(N),M)\rightarrow
\mathrm{Ext}^{i}_R(N,M)\rightarrow\mathrm{Ext}^{i}_R(\Gamma_\mathfrak{a}(N),M)$,\end{center}
\begin{center}$\mathrm{Tor}_{i}^R(\Gamma_\mathfrak{a}(N),M)\rightarrow\mathrm{Tor}_{i}^R(N,M)\rightarrow
\mathrm{Tor}_{i}^R(N/\Gamma_\mathfrak{a}(N),M)\rightarrow\mathrm{Tor}_{i-1}^R(\Gamma_\mathfrak{a}(N),M)$.\end{center}
We may assume $\Gamma_\mathfrak{a}(N)=0$. Then $\mathfrak{a}\nsubseteq\bigcup_{\mathfrak{p}\in\mathrm{Ass}_RN}\mathfrak{p}$ by \cite[Lemma 2.1.1]{BS}, and there exists an element $x\in \mathfrak{a}$ and an exact sequence
$0\rightarrow N\stackrel{x}\rightarrow N\rightarrow N/xN\rightarrow 0$, which
induces the following exact sequence
\begin{center}$\mathrm{Ext}^i_R(N/xN,M)\rightarrow\mathrm{Ext}^i_R(N,M)\stackrel{x}\rightarrow \mathrm{Ext}^i_R(N,M)\rightarrow\mathrm{Ext}^{i+1}_R(N/xN,M)$,\end{center}
\begin{center}$\mathrm{Tor}_{i+1}^R(N/xN,M)\rightarrow\mathrm{Tor}_{i}^R(N,M)\stackrel{x}\rightarrow \mathrm{Tor}_{i}^R(N,M)\rightarrow\mathrm{Tor}_{i}^R(N/xN,M)$\end{center}
for all $i\geq0$. Hence we have an exact sequence
$\mathrm{Ext}^i_R(N/xN,M)\rightarrow (0:_{\mathrm{Ext}^i_R(N,M)}x)\rightarrow 0$ and $\mathrm{Tor}_{i+1}^R(N/xN,M)\rightarrow (0:_{\mathrm{Tor}_{i}^R(N,M)}x)\rightarrow 0$ for $i\geq0$.
As the $R$-module $N/xN$ is of finite length, $(0:_{\mathrm{Ext}^i_R(N,M)}x),(0:_{\mathrm{Tor}_{i}^R(N,M)}x)\in\mathcal{NS}$ by Lemma \ref{lem:3.5'} and $\mathrm{dim}_R(0:_{\mathrm{Ext}^i_R(N,M)}x)=0=\mathrm{dim}_R(0:_{\mathrm{Tor}_{i}^R(N,M)}x)$ for all $i\geq0$. Thus
$(0:_{\mathrm{Ext}^i_R(N,M)}\mathfrak{a}),(0:_{\mathrm{Tor}_{i}^R(N,M)}\mathfrak{a})\in\mathcal{NS}$ and $\mathrm{dim}_R(0:_{\mathrm{Ext}^i_R(N,M)}\mathfrak{a})=0=\mathrm{dim}_R(0:_{\mathrm{Tor}_{i}^R(N,M)}\mathfrak{a})$ for all $i\geq0$. Note that $\mathrm{Supp}_R\mathrm{Ext}^i_R(N,M)$, $\mathrm{Supp}_R\mathrm{Tor}_{i}^R(N,M)\subseteq\mathrm{V}(\mathfrak{a})$, we have $\mathrm{dim}_R\mathrm{Ext}^i_R(N,M)=0=\mathrm{dim}_R\mathrm{Tor}_{i}^R(N,M)$. Hence Lemma \ref{lem:2.1} implies that $\mathrm{Ext}^i_R(N,M)$ and $\mathrm{Tor}_{i}^R(N,M)$ are $\mathcal{NS}$-$\mathfrak{a}$-cofinite for all $i\geq 0$.
\end{proof}

\begin{lem}\label{lem:3.8}{\it{The class of $\mathcal{NS}$-$\mathfrak{a}$-cofinite $R$-modules of zero dimension is closed under taking submodules and quotients.}}
\end{lem}
\begin{proof} Let $0\rightarrow L\rightarrow M\rightarrow N\rightarrow 0$ be a short exact of $R$-modules with $M$ $\mathcal{NS}$-$\mathfrak{a}$-cofinite of zero dimensions. Then $\mathrm{Hom}_R(R/\mathfrak{a},L)\in\mathcal{NS}$. So $L$ is $\mathcal{NS}$-$\mathfrak{a}$-cofinite by Lemma \ref{lem:2.1} and therefore $N$ is $\mathcal{NS}$-$\mathfrak{a}$-cofinite.
\end{proof}

The next main theorem of this section generalizes \cite[Theorems 2.8 and 2.10]{S} and \cite[Theorem 2.4]{AB} and \cite[Theorem 2.4]{NBG}.

\begin{thm}\label{lem:3.7}{\it{Let $M$ be an $\mathcal{NS}$-$\mathfrak{a}$-cofinite $R$-module and $N$ a finitely generated $R$-module with $\mathrm{dim}_RN\leq2$. Then $\mathrm{Ext}^i_R(N,M)$ and $\mathrm{Tor}_{i}^R(N,M)$ are $\mathcal{NS}$-$\mathfrak{a}$-cofinite for all $i\geq0$.}}
\end{thm}
\begin{proof} By analogy with the proof of Lemma \ref{lem:3.6}, we may assume $\Gamma_\mathfrak{a}(N)=0$ and
$\mathrm{dim}_RN=2$. Then there exists an element $x\in \mathfrak{a}$ and an exact sequence
$0\rightarrow N\stackrel{x}\rightarrow N\rightarrow N/xN\rightarrow 0$, which
induces two exact sequences
\begin{center}$\mathrm{Ext}^i_R(N/xN,M)\rightarrow\mathrm{Ext}^i_R(N,M)\stackrel{x}\rightarrow \mathrm{Ext}^i_R(N,M)\rightarrow\mathrm{Ext}^{i+1}_R(N/xN,M)$,\end{center}
\begin{center}$\mathrm{Tor}_{i+1}^R(N/xN,M)\rightarrow\mathrm{Tor}_{i}^R(N,M)\stackrel{x}\rightarrow \mathrm{Tor}_{i}^R(N,M)\rightarrow\mathrm{Tor}_{i}^R(N/xN,M)$\end{center}
for $i\geq0$. Since $\mathrm{dim}_RN/xN=1$, it follows from Lemma \ref{lem:3.6} that $\mathrm{Ext}^i_R(N/xN,M)$ and $\mathrm{Tor}_{i}^R(N/xN,M)$ are $\mathcal{NS}$-$\mathfrak{a}$-cofinite of zero dimension. Thus
$(0:_{\mathrm{Ext}^i_R(N,M)}x), (0:_{\mathrm{Tor}_{i}^R(N,M)}x)$ and $\mathrm{Ext}^i_R(N,M)/x\mathrm{Ext}^i_R(N,M),\mathrm{Tor}_{i}^R(N,M)/x\mathrm{Tor}_{i}^R(N,M)$ are $\mathcal{NS}$-$\mathfrak{a}$-cofinite
 for all $i\geq0$ by Lemma \ref{lem:3.8}. Consequently, by \cite[Lemma 2.2]{AS}, the $R$-modules $\mathrm{Ext}^i_R(N,M)$ and $\mathrm{Tor}_{i}^R(N,M)$ are $\mathcal{NS}$-$\mathfrak{a}$-cofinite
for all $i\geq0$.
\end{proof}

The following result is a generalization of \cite[Theorems 2.5 and 2.10]{AB}.

\begin{cor}\label{lem:4.7}{\it{Let $(R,\mathfrak{m})$ be local, and let $M$ be an $\mathcal{NF}$-$\mathfrak{a}$-cofinite $R$-module and $N$ a finitely generated $R$-module such that either $\mathrm{dim}_RM=2$ or $\mathrm{dim}_RN=3$. Then $\mathrm{Ext}^i_R(N,M)$ and $\mathrm{Tor}_{i}^R(N,M)$ are $\mathcal{NF}$-$\mathfrak{a}$-cofinite for all $i\geq0$.}}
\end{cor}
\begin{proof} Denote $\Phi$ the set of all modules $\mathrm{Ext}_R^j(R/\mathfrak{a},\mathrm{Ext}^i_R(N,M))$ and $\mathrm{Ext}_R^j(R/\mathfrak{a},\mathrm{Tor}_i^R(N,M))$ for $i,j\geq0$. Let $L\in\Phi$ and $L'$ be a submodule of $L$. It is enough to
show that $\mathrm{Ass}_RL/L'$ is finite. To this end, according to \cite[Exercise 7.7]{M} and \cite[Lemma 2.1]{TM} we may assume that $R$ is complete. Suppose the contrary is true. Then
there exists a countably infinite subset $\{\mathfrak{p}_k\}^\infty_{k=1}$ of $\mathrm{Ass}_RL/L'$, such that none of which is not equal
to $\mathfrak{m}$, and hence $\mathfrak{m}\nsubseteq\bigcup_{k=1}^\infty\mathfrak{p}_k$ by \cite[Lemma 3.2]{TM}. Let $S=R\backslash\bigcup_{k=1}^\infty\mathfrak{p}_k$. Then
the $S^{-1}R$-module $S^{-1}M$ is $\mathcal{NF}$-$S^{-1}\mathfrak{a}$-cofinite with $\mathrm{dim}_{S^{-1}R}S^{-1}M\leq1$ or $\mathrm{dim}_{S^{-1}R}S^{-1}N\leq2$, it follows from Corollary \ref{lem:2.7'} and Theorem \ref{lem:3.7}
that $S^{-1}L$ is a weakly Laskerian $S^{-1}R$-module and so $\mathrm{Ass}_{S^{-1}R}(S^{-1}L/S^{-1}L')$ is a finite set. But
$S^{-1}\mathfrak{p}_k\in\mathrm{Ass}_{S^{-1}R}(S^{-1}L/S^{-1}L')$  for all $k=1,2,\cdots$, which is a contradiction.
\end{proof}

\end{document}